\documentclass{amsart}
\usepackage{amsfonts}
\usepackage{amssymb}
\usepackage{dsfont}
\usepackage{amsmath, amsthm}
\usepackage{tikz}
\usepackage{graphicx, color}
\usepackage{ mathrsfs }
\newtheorem{mtheorem}{Theorem}
\newtheorem{theorem}{\textbf{Theorem}}[section]

\newtheorem{claim}[theorem]{\textbf{Claim}}

\newtheorem{corollary}{\textbf{Corollary}}

\newtheorem{definition}[theorem]{\textbf{Definition}}
\newtheorem{lemma}[theorem]{\textbf{Lemma}}

\newtheorem{remark}[theorem]{\textbf{Remark}}
\newtheorem{assumption}{\textbf{Assumption}}

\newtheorem*{question*}{\textbf{Question}}
\theoremstyle{plain}

\newcommand{\eqdef}{\stackrel{\scriptscriptstyle\rm def}{=}}

\definecolor{Red}{cmyk}{0,1,1,0}

\definecolor{Blue}{cmyk}{1,1,0,0}

\newtheoremstyle{example}
  {}
  {}
  {}
  {}
  {\itshape}
  {.}
  {.5em}
  {\thmname{#1}\thmnumber{ #2}\thmnote{ (#3)}}
\theoremstyle{example}

\numberwithin{equation}{section}

\title[Random iterations of maps]
{Random iterations of maps on $\mathbb{R}^{k}$: asymptotic stability, synchronization  and functional central limit theorem}
\author[E. Matias and E. A. Silva]{Edgar Matias and Eduardo Silva }
\address{Departamento de Matem\' atica, Universidade Federal da Bahia, Av. Adhemar de Barros s/n, 40170-110 Salvador, Brazil}
\email{edgar.matias@ufba.br}

\address{Departamento de Matemática, Universidade de Bras\'ilia, 70910-900, Bras\'ilia, Brazil}
\email{e.a.silva@mat.unb.br} 
\begin{document}

 
 
 \begin{abstract}
 We study independent and identically distributed  random iterations of continuous maps defined on a connected closed subset $S$ of the Euclidean space $\mathbb{R}^{k}$. We assume the maps are monotone (with respect to a suitable partial order) and a ``topological'' condition on the maps. Then, we prove the existence of a pullback random attractor whose distribution is the unique stationary measure of the random iteration, and we obtain the synchronization of random orbits. As a consequence of the synchronization phenomenon, a functional central limit theorem is established.  
 
 \end{abstract} 
 
\begin{thanks}{
 E. Matias and E. A. Silva  are  supported by FAPDF. The authors thank Mariela Pent\' on for her useful comments.}\end{thanks}
 
%
 
%

\maketitle
\section{Introduction}
In this paper, we study Markov chains through its embedding on a discrete random dynamical system with white noise. This approach goes back to Furstenberg in his pioneer work on random products, see \cite{Fur}.
 Namely, let $X=\{X_{n}\}$ be an independent and identically distributed (i.i.d.) sequence of random variables taking values on  a measurable space $E$ and consider a family $\{f_{\alpha}\}_{\alpha\in E}$ of maps $f_{\alpha}\colon S\to S$. Under appropriate measurability assumptions,  these two ingredients specify a homogeneous Markov chain with state space $S$ given by 
\begin{equation}\label{i.i.d.}
Z_{n}=f_{X_{n-1}}\circ \dots \circ f_{X_{0}}(Z_{0})
\end{equation}
called an \emph{i.i.d. random iteration of maps}, where $Z_{0}$ is a random variable independent of $X=\{X_{n}\}$ taking values 
on $S$.  It is well known that any Markov chain on a standard measurable space admits a representation as in \eqref{i.i.d.}, see Kifer \cite{Kifer}. 
%
%

In this setting, Markov chains have been extensively investigated. It is worth mentioning that there is no unifying technique to be applied in this study. They vary according to the class of maps chosen.  A notable example is the seminal work of Hutchinson \cite{Hu}. Therein, the author considers a finite number of contractions maps on a complete metric space and shows that the induced transition probability is asymptotically stable, i.e., that there exists a unique initial distribution (the distribution of $Z_{0}$) for which the Markov chain $Z_{n}$ is stationary, and under every initial distribution the sequence $Z_{n}$ converges in distribution to the stationary measure. The topological support of the stationary measure is a compact set called the \emph{Hutchinson attractor} and plays a fundamental role in the study of fractals. The development of this theory has found applications in diverse areas such as computer graphics, image compression, and fractal theory. 

From the probabilistic point of view, the results of the ``Hutchinson theory'' were generalized for contracting on average systems. Namely, the asymptotic stability was proved for i.i.d. random iterations of finitely many Lipschitz maps having negative extremal Lyapunov exponent in \cite{Bar1}. Later a generalization for infinitely many maps was given in \cite{Diaconis}. See also \cite{placedependent} for a generalization for iterated function systems with place dependent probabilities and \cite{Stark} for related results on stationary random iterations of Lipschitz maps. Also, we observe that a central limit theorem (CLT) was obtained in some cases, see  \cite{Benda, Stenflo, Biao}.
  
  In this paper, we study statistical properties of $Z_{n}$ for a certain class of continuous monotone maps on $\mathbb{R}^{k}$ with respect to suitable partial orders. Two of our main results are asymptotic stability and a functional central limit theorem (FCLT), see Theorems \ref{pullback} and \ref{functionalclt}. We observe that  an analysis of extremal Lyapunov exponents does not take place here. Indeed, we do not even assume the maps are Lipschitz. 

   The study of i.i.d. random iterations of monotone maps goes back to Dubins and Friedman \cite{Dubins}, where they consider continuous monotone maps on the unit interval and prove the asymptotic stability assuming a condition called \emph{splitting condition}. A generalization of this result in higher dimensions was given in \cite{BhLe} for  a class of non-decreasing continuous maps (w.r.t. the componentwise order) defined on a closed subset $S\subset\mathbb{R}^{k}$. Therein, the authors also state a functional central limit theorem in higher dimensions. However, the proof presented for the FCLT works only on dimension one. A corrected proof was given two decades later in \cite{Bhbook} under a stronger hypothesis than the previous one.  These results have found applications in mathematical economics and nonlinear autoregressive models, see \cite{Bhbook, BhLe}.

  Our paper partially follows the same line of \cite{BhLe} in the sense that the asymptotic stability and the FCLT  are also stated for an appropriate extension for higher dimensions of the splitting condition introduced in \cite{Dubins}. We consider a generalization of the strict componentwise order for which the induced family of monotone maps widely enlarge the class of monotone maps with respect to the strict componentwise order. However, the ideas from \cite{BhLe} can not be adapted to this context. To overcome this situation we follow a dynamical system approach. It is noteworthy that we established an FCLT for a family consisting of both increasing and decreasing maps, while in \cite{Bhbook} the techniques work only for families of non-decreasing maps with respect to the componentwise order.

Our strategy consists in studying the topological behavior of  $Z_{n}$ and its ``dual'' $\hat Z_{n}$, a sequence of random variables defined by taking reverse order iterations
$$
\hat Z_{n}=f_{X_{0}}\circ \dots \circ f_{X_{n-1}}(Z_{0}).
$$
A relation between asymptotic stability and the topological convergence of the sequence $\hat Z_{n}$ is given by a well-known result called the Letac principle.  See \cite{Letac} for more details.  On the other hand, the functional central limit theorem is obtained through an analysis of the dynamical behavior of $Z_{n}$. More precisely, we investigate the \emph{synchronization phenomenon} for i.i.d. random iterations of monotone maps, see Theorem \ref{sync}.
 
  The synchronization phenomenon was first observed  by  Huygens \cite{Huygens} in the movement of two pendulum clocks hanging from a wall and since then has been investigated in several areas, see \cite{syncrhonizationbook}.
 For random iterations, results on synchronization were obtained in several settings
 where no contraction-like property is given a priori, see \cite{Antonov, Goro,Homburg, KlepAntonov,Malicet}.

  The synchronization in a random environment manifests itself in several ways, whether local, in mean, or global. In this paper, the synchronization appears globally, meaning that the topological behavior of $Z_{n}^{x}$ is asymptotically the same for all $x\in S$, where  $Z_{n}^x$ denotes the Markov chain in \eqref{i.i.d.} with initial distribution $Z_{0}=x$.
 In other words, for every pair $x,y\in S$, the i.i.d. random iterations $Z_{n}^{x}$ and $ Z^{y}_{n}$  satisfy
$$
d(Z^{x}_{n},  Z^{y}_{n})\underset{n\to \infty}\longrightarrow 0
$$
with probability $1$, where $d$ is the distance in $S$.

   The synchronization of i.i.d. random iterations, or some contraction property, usually leads to a central limit theorem (CLT). In \cite{Benda, Biao},  central limit theorems are obtained for a certain class of i.i.d. random iterations of Lipschitz maps having negative Lyapunov exponent.  For an i.i.d. random iteration of homeomorphisms on the circle, Malicet in \cite{Malicet} shows that a local  
synchronization holds under the assumption that the maps do not have an invariant measure in common. This result was later used to prove a CLT in \cite{Annaa}. In \cite{Stenflo}, a central limit theorem is obtained for contractive iterated function systems with place-dependent probabilities.

Most of the synchronization properties in the literature are derived from Lyapunov exponents, with a suitable definition in each setting. However, in some situations Lyapunov exponents play no role. For instance,  this is the case of the i.i.d. random iteration of double rotations studied in \cite{Goro}, where the authors obtain a synchronization on average using properties of simple random walks on $\mathbb{Z}$. In our paper, we also
 do not use Lyapunov exponents. Instead,  we explore our extension of  the splitting condition to show the shrinking of the entire space under the action of the random compositions of the monotone maps.
 
Finally, for the class of i.i.d. random iterations of monotone maps that we consider in this paper, as a simple consequence of our study of the sequence $\hat Z_{n}$, we characterize stationary measures  as the distribution of pullback random attractors, a notion of random attractors introduced by Crauel and Flandoli in \cite{CrauelFranco}.

 We refer to \cite{Recurrent, EdgarD,Volk} for related results for Markovian random iterations of finitely many maps.

\subsection*{Organization of the paper}

In Section \ref{mainn} we state precisely the main definitions and results of this work. In Section \ref{p1} we prove Theorem \ref{pullback} and Corollary \ref{Cor1}. In Section \ref{p2} we prove Theorem \ref{sync} and Corollaries \ref{expofastW} and  \ref{Cor2}. Theorem \ref{functionalclt} is proved in Section \ref{p3}.
 
\section{Statements of results}\label{mainn}

\subsection{General setting}\label{GS}

Let $(E,\mathscr{F}, \nu)$ be a probability space and consider an i.i.d. sequence of random variables $X=\{X_{n}\}$ with state space $E$. Throughout, $\nu$ denotes the common distribution of $X$ and $(\Omega, \mathscr{F},\mathbb{P})$ is the probability space where $X$ is defined.
Let $S\subset \mathbb{R}^{k}$ be a connected subspace and  consider a measurable map $f\colon E\times S\to S$. We denote by $f_{\alpha}$ the map $f_{\alpha}(x)=f(\alpha,x)$.
For every random variable $Z_{0}\colon \Omega\to S$ independent of $X$, the pair $(f,X)$ induces a homogeneous Markov chain $Z_{n}$ as defined in \eqref{i.i.d.}, whose transition probability is given by
\begin{equation}\label{tp}
p(x, A)=\int \mathds{1}_{A}(f_{\alpha}(x))\, d\nu(\alpha).
\end{equation}
The pair $(f,X)$, as well as any of its induced Markov chain, will be called an \emph{i.i.d. random iteration of maps}. 

Now, let us introduce the class of maps that we study in this paper. 
Consider a subset $J\subset\{1,\dots, k\}$ and define a partial order as follows: given $x,y\in \mathbb{R}^{k}$, we write $x<_{J} y$ if and only if 
$$
x_{i}<y_{i} \,\,\, \mbox{ for }\,\,\, i\in J \,\,\,\mbox{ and }\,\,\, x_{i}>y_{i} \,\,\, \mbox{ for }\,\,\, i\notin J.
$$
Note that if $J=\{1,\dots, k\}$, then the induced order is the well-known strict componentwise order.

 We  
say that a map $f\colon S\to S$  is  $J$-\emph{increasing}  if 
$$
x<_{J}y \Rightarrow f(x)<_{J}f(y).
$$
Similarly, we say that $f$ is $J$-\emph{decreasing} if 
$$
x<_{J}y \Rightarrow f(y)<_{J}f(x).
$$
The map $f$ is called  $J$-\emph{monotone} if $f$ is either  $J$-increasing or $J$-decreasing.

Given subsets $S_{1}, S_{2}\subset \mathbb{R}^{k}$, we write $S_{1}<_{J} S_{2}$ if  
$$
x<_{J}y\,\,\, \mbox{for every} \,\,\, x\in S_{1}\,\,\,\mbox{and}\,\,\, y\in S_{2}.
$$ 
Let $\pi_{s}\colon \mathbb{R}^k\to \mathbb{R}$ be the natural projection $\pi_{s}(x)=x_{s}$, $s=1,\dots, k$. 

\begin{remark}\label{capempty}
\emph{
A crucial property of the partial order $<_{J}$ for the proofs of our main theorems is the following:
   if $S_{1}<_{J} S_{2}$, then 
$$
\pi_{s}(S_{1})\cap \pi_{s}(S_{2})=\emptyset
$$
for every $s=1,\dots, k$. In other words, the respective projections of $S_{1}$ and $S_{2}$ are disjoints.  In particular, for every $J$-monotone map $f$ we have that $f(S_{1})$ and $f(S_{2})$ have disjoints projections because either $f(S_{1})<_{J}f(S_{2})$ or $f(S_{2})<_{J}f(S_{1})$.} 
\end{remark} 
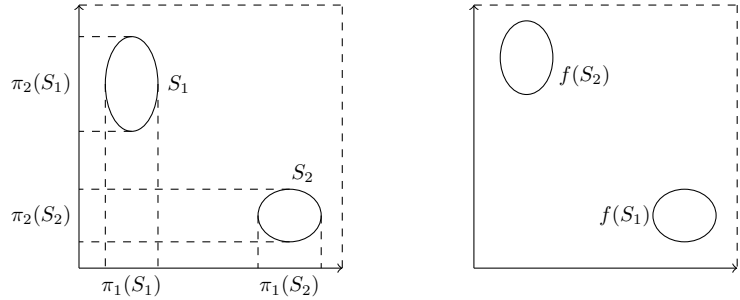
\begin{figure}[h!]
\centering
\begin{tikzpicture}[xscale=3.5,yscale=3.5]
\draw[->] (0,0)-- (1,0);
\draw[->] (0,0)--(0,1);
\draw (0.2,0.7) ellipse (1mm and 1.8mm);
\draw (0.8,0.2) ellipse (1.2mm and 1mm);
\draw[dashed] (0,1)--(1,1);

\draw[dashed] (1,0)--(1,1);
\draw[dashed] (0.2,0.88)--(0,0.88);
\draw[dashed] (0.2,0.52)--(0,0.52);

\draw[dashed] (0.8,0.3)--(0,0.3);
\draw[dashed] (0.8,0.1)--(0,0.1);

 \node[scale=0.8,left] at (0.45,0.7) {$S_{1}$};
 
 \node[scale=0.8,above] at (0.85,0.3) {$S_{2}$};
 
 \node[scale=0.8,left] at (2.2,0.2) {$f(S_{1})$};
 
 \node[scale=0.8,below] at (1.92,0.8) {$f(S_{2})$};

 \node[scale=0.8,below] at (0.8,0) {$\pi_{1}(S_{2})$};
\node[scale=0.8,below] at (0.2,0) {$\pi_{1}(S_{1})$};

 \node[scale=0.8,left] at (0,0.2) {$\pi_{2}(S_{2})$};
\node[scale=0.8,left] at (0,0.7) {$\pi_{2}(S_{1})$};

\draw[dashed] (0.1,0.7)--(0.1,0);
\draw[dashed] (0.3,0.7)--(0.3,0);

\draw[dashed] (0.68,0.2)--(0.68,0);
\draw[dashed] (0.92,0.2)--(0.92,0);
\draw[->] (1.5,0)-- (2.5,0);
\draw[->] (1.5,0)--(1.5,1);

\draw (1.7,0.8) ellipse (1mm and 1.4mm);
\draw (2.3,0.2) ellipse (1.2mm and 1mm);

\draw[dashed] (1.5,1)--(2.5,1);
\draw[dashed] (2.5,1)--(2.5,0);
\end{tikzpicture}
\caption{The action of a $J$-decreasing map on $\mathbb{R}^{2}$ for $J=\{1\}$. Take, for instance, $f(x,y)=(\arctan(y-x), e^{x-y})$.}
\label{f.nonregular}
\end{figure}

In this paper, we study i.i.d. random iterations of $J$-monotone maps satisfying the following ``topological'' property:    
\begin{definition}\label{Definition1}
\emph{ We say that an i.i.d. random iteration $(f,X)$ of $J$-monotone maps satisfies the $J$-\emph{splitting condition} if 
 there are $m\in \mathbb{N}$ and measurable subsets $A$ and $B$ of $E^m$ with 
$\nu^m(A)>0$ and $\nu^m(B)>0$ such that for every $(\alpha_{0},\dots, \alpha_{m-1})\in A$ and 
$(\beta_{0},\dots, \beta_{m-1})\in B$ we have
$$
 f_{\alpha_{0}}\circ\dots  \circ  f_{\alpha_{m-1}}(S)<_{J} f_{\beta_{0}}\circ\dots  \circ f_{\beta_{m-1}}(S).
 $$}
\end{definition}
This condition is a generalization of the splitting condition introduced in \cite{Dubins} and is closely related to the splitting condition considered in \cite{Bhbook, BhLe}, see the discussion in Section \ref{S.stationary}.

\subsection{Stationary measures}\label{S.stationary}
 Let $p$ be the transition probability defined in \eqref{tp}. Associate with 
 $p$ there is an operator acting on the space of probability measures on $S$ given by
 $\mu\mapsto T\mu$, where $T\mu$ is the probability measure defined by   
 $$
T\mu(A)=\int p(x,A)\, d\mu(x)
$$ 
for every Borel set $A\subset S$.
A fixed point for $T$ is called a \emph{stationary measure}. In other words, a probability measure $\mu$ on $S$ is a stationary measure if 
 \[
 \mu(A)=\int p(x,A)\, d\mu(x) 
 \]
 for every Borel set $A\subset S$.  We say that $T$ is \emph{asymptotically stable} if there is a stationary measure $\mu$ such that for every probability measure $\varsigma$ we have that $T^{n}\varsigma$ converges to $\mu$ in the weak-star topology.

 Results on the stability of Markov operators for i.i.d. random iterations of monotone continuous maps go back to Dubins and Freedman \cite{Dubins}. Therein,  asymptotic stability is proved for Markov operators associated with  i.i.d. random iterations of monotone maps on $[0,1]$ satisfying a condition called \emph{splitting condition}: 
there are $x_{0}\in \mathbb{R}$ and $m\geq 1$ such that 
\begin{equation}\label{bata}
\mathbb{P}(Z_{m}^x\leq x_{0}\,\, \forall x)>0 \quad \mbox{and}\quad \mathbb{P}(Z_{m}^x\geq x_{0}\,\, \forall x)>0.
\end{equation}

 This result was generalized in higher dimensions in \cite{Bhbook,BhLe} for i.i.d. random iterations of monotone maps defined on a closed subset $S\subset \mathbb{R}^{k}$ satisfying the multidimensional analog of the splitting condition, where the total order $\leq$ of $\mathbb{R}$ is replaced by the componentwise order on $\mathbb{R}^k$. Recall that the componentwise order is a partial order $\leq $ defined by: $x\leq y$ if and only if $x_{i}\leq y_{i}$ for every $i=1,\dots ,k$. It is not difficult to see that in the case $<_{J}$ is the strict componentwise order (i.e., $J=\{1,\dots, k\}$), a $J$-monotone map is also a monotone map with respect to the componentwise order and an i.i.d. random iteration of $J$-monotone maps satisfying the $J$-splitting condition also satisfy the splitting condition with respect to the componentwise order $\leq$ as defined in \eqref{bata}.

 Our first result states the asymptotic stability of Markov operators associated with i.i.d. random iterations of $J$-monotone continuous maps satisfying the $J$-splitting condition, extending the results on asymptotic stability in \cite{Bhbook,BhLe} for a considerable class of monotone maps.
 

  In what follows, 
 for any measurable map $\pi\colon \Omega\to S$ we denote by $\pi\mathbb{P}$ the image of $\mathbb{P}$ by $\pi$, i.e., the probability measure on $S$ given by $\pi\mathbb{P}(A)=\mathbb{P}(\pi^{-1}(A))$ for every Borel set $A\subset S$. The probability measure $\pi\mathbb{P}$ is also called the distribution of $\pi$.

%
%

\begin{mtheorem}\label{pullback}
Let $S$ be a connected closed subset of $\mathbb{R}^k$ and let  $(f,X)$ be an i.i.d. random iteration of $J$-monotone continuous maps on $S$ satisfying the $J$-splitting condition. Then
\begin{enumerate}
\item [(i)]
 There is a measurable map $\pi\colon \Omega\to S$ such that
for $\mathbb{P}$-almost every $\omega$ we have 
\[
\pi(\omega)=\lim_{n\to \infty} f_{X_{0}(\omega)}\circ\dots\circ f_{X_{n}(\omega)}(x)
\]
for every $x\in S$.
\item[(ii)]The probability measure $\pi\mathbb{P}$ is the unique stationary measure and  for every probability measure $\varsigma$ on $S$ we have
$$
T^n\varsigma\to \pi\mathbb{P}
$$
in the weak-star topology, where $T$ is the Markov operator.
\end{enumerate}

\end{mtheorem}

The second item of Theorem \ref{pullback} says that the Markov operator is asymptotically stable. We observe that the proof  of item (ii) follows from item (i) and the Letac principle, see 
\cite{Letac}.

\subsubsection{Pullback attractor}\label{Apullbak} In this section, we take an alternative point of view of i.i.d. random iterations in order to show that the map $\pi$ of Theorem \ref{pullback} is a \emph{pullback random attractor}. We observe that this notion and several notions of random attractors have been extensively studied by Crauel et al. in \cite{CrauelHans,CrauelHans2, CrauelHans3,CrauelDebussche,CrauelDimitroff, CrauelFranco}. See also Section \ref{forwardattractor}.

Let $\Omega=E^ {\mathbb{Z}}$ be the product space endowed with the product $\sigma$-algebra 
 and the product measure $\mathbb{P}=\nu^{\mathbb{Z}}$. Consider a measurable map $f\colon E\times S\to S$ and denote by $f_{\alpha}$ the map $f_{\alpha}(x)=f(\alpha,x)$. Following Arnold \cite{Arnold}, the  map $\varphi\colon \mathbb{N}\times \Omega\times S\to S$ given by 
 \begin{equation}\label{newphi}
 \varphi(n,\omega,x)=f_{\omega_{n-1}}\circ \dots \circ f_{\omega_{0}}(x)\eqdef f_{\omega}^{n}(x)
 \end{equation}
 is called an \emph{i.i.d. random iteration of maps}.

Note that the sequence of natural projections $X_{n}(\omega)=\omega_{n}$, $n\geq 0$, is 
an i.i.d. sequence of random variables with distribution $\nu$. Hence the sequence $\omega\mapsto f_{\omega}^{n}(x)$ is also an i.i.d. random iteration in the sense of \eqref{i.i.d.}.
We say that the map $\varphi$ satisfies the $J$-splitting condition if $(f,X)$ does satisfy.

We recall that a (global point) \emph{pullback random attractor} of $\varphi$ is a $\varphi$-invariant random compact set $\omega\mapsto K(\omega)$ such that for every $x\in S$ 
$$
\lim_{n\to \infty}d(\varphi(n,\sigma^{-n}(\omega),x),K(\omega))=0, 
$$ 
for $\mathbb{P}$-almost every $\omega$,
where $\sigma$ is the shift map on $\Omega$. See for instance Crauel and Scheutzow \cite{CrauelScheutzow}.


As a consequence of Theorem \ref{pullback} we have:
\begin{corollary}\label{Cor1}
Let $S$ be a connected closed subset of $\mathbb{R}^k$ and let  $\varphi$ be an i.i.d. random iteration of $J$-monotone continuous maps on $S$ satisfying the $J$-splitting condition. Then,
there is a measurable map $\pi\colon \Omega\to S$ such that for every $x\in S$
$$
\lim_{n\to \infty}d(\varphi(n,\sigma^{-n}(\omega),x),\pi(\omega))=0
$$
for $\mathbb{P}$-almost every $\omega$. Moreover,  the distribution of $\pi$ is the unique stationary measure.  
\end{corollary}
 Corollary \ref{Cor1} says that the random compact set $\omega\mapsto \{\pi(\omega)\}$ is a pullback random attractor. As an example, consider the maps $f_{1}(x)=e^x$ and $f_{2}(x)=-e^{x}$ on $\mathbb{R}$.
We take $E=\{1,2\}^{\mathbb{Z}}$ and $\mathbb{P}=\nu^{\mathbb{Z}}$, where $\nu$ is a probability measure on $\{1,2\}$ such that $\nu(\{i\})>0$ for $i=1,2$. 
Then, the i.i.d. random iteration $\varphi$ given by 
$$
\varphi(n,\omega,x)=f_{\omega_{n-1}}\circ \dots \circ f_{\omega_{0}}(x)
$$
satisfies the $J$-splitting condition. 

\subsection{Synchronization}\label{syncq}

We say that  an i.i.d. random iteration of maps $(f,X)$ satisfy the \emph{synchronization property} if for every $x$ and $y$, 
there is a measurable set $\Omega^{x,y}\subset \Omega$ with $\mathbb{P}$-full measure, such that
$$
\lim_{n\to \infty}d(Z_{n}^{x}(\omega),Z_{n}^{x}(\omega))=0
$$  
for every $\omega\in \Omega^{x,y}$. In other words, given any $x$ and $y$, with probability 
$1$, the random orbits $Z_{n}^{x}$ and  $Z_{n}^{y}$ converge to each other.
In what follows, we use  
$f_{\omega}^{n}(x)$ to denote $Z_{n}^{x}(\omega)$.


The class of i.i.d. random iterations of Lipschitz maps with negative maximal Lyapunov exponent is the most known class satisfying the synchronization property. Namely, let $(f,X)$ be an i.i.d. random iteration of Lipschitz maps and assume that there is an integrable map $c\colon \Omega\to \mathbb{R}$ such that for 
$\mathbb{P}$-almost every $\omega$ we have
$$
d(f_{\alpha}(x),f_{\alpha}(y))\leq c(\omega)d(x,y)
$$
for every $x,y\in S$. The \emph{maximal Lyapunov exponent} of the i.i.d.  random iteration $(f,X)$ is defined by 
$$
\lambda(\omega)=\lim_{n\to\infty} \frac{1}{n}\log \sup_{x,y}\frac{d(f^{n}_{\omega}(x),f^{n}_{\omega}(y))}{d(x,y)}.
$$
The above limit exits for $\mathbb{P}$-almost every $\omega$ by Kingman's theorem, see \cite{Stark}. Note that if the Lyapunov exponent is bounded (a.e.) by a negative constant, 
then it follows from the definition that there are a measurable map $C\colon \Omega\to \mathbb{R}$ and $\lambda<1$ such that 
$$
d(f_{\omega}^{n}(x),f^{n}_{\omega}(y))\leq C(\omega) \lambda^n d(x,y).
$$
In other words, negative maximal Lyapunov exponent implies exponentially fast synchronization. A practical way to verify the negativity of the Lyapunov exponent is the following estimate 
\begin{equation}\label{folk}
\lambda(\omega)\leq \int \log c(\omega)\, d\mathbb{P}(\omega)\eqdef \lambda_{0}
\end{equation}
for $\mathbb{P}$-almost every $\omega$. This seems to be a folklore result and we do not found a reference stating \eqref{folk} explicitly. We observe that this estimate holds for any \emph{stationary} random iteration of Lipschitz maps (not necessarily i.i.d. random iterations). A proof can be performed using ideas from \cite[Section 5]{Stark}.


In our next theorem, we present a result stating exponentially fast synchronization for a certain class of i.i.d. random iterations of $J$-monotone continuous maps. Under an additional boundedness condition on the maps we show that the $J$-splitting condition implies (uniform) synchronization. We observe that in our study a negative extremal Lyapunov exponent is replaced by the $J$-splitting condition.

 
\begin{assumption}\label{A2}
There are a bounded set $B$ and $m_{0}$ such that  
$$
f_{\omega}^{m_{0}}(S)\subset B 
$$
for $\mathbb{P}$-almost every $\omega$.
\end{assumption}

Clearly, Assumption \ref{A2} holds when either $S$ is bounded or the images of the maps $f_{\alpha}$ are bounded.
For an example where neither $S$ is bounded nor the images of the maps are bounded, consider the maps $f_{1},f_{2}\colon \mathbb{R}\to \mathbb{R}$ where $f_{1}(x)=e^{-x}$ and $f_{2}$ is a strictly monotone map whose image is a bounded subset $B_{0}$ of $(-\infty,0)$. Then, for every $i,j\in\{1,2\}$ the composition $f_{i}\circ f_{j}$ has image contained in $B\eqdef B_{0}\cup [0,1]$.

\begin{mtheorem} \label{sync} Let $S$ be a connected subset  of $\mathbb{R}^k$ and let  $(f,X)$ be an i.i.d. random iteration of $J$-monotone continuous maps on $S$ satisfying the $J$-splitting condition. Then, under Assumption \ref{A2}, there exist constants $r<1$ and $m_{0}\geq 1$,  and  an integrable map $c\colon \Omega\to [0,\infty)$ such that  for $\mathbb{P}$-almost every $\omega$
\[
\emph{diam}\,f_{\omega}^{n}(S)\leq c(\omega)r^n
\]
for every $n\geq m_{0}$.
\end{mtheorem}
 
Let us illustrate Theorem \ref{sync} presenting an example of i.i.d. random iterations of Lipschitz maps satisfying the synchronization property, which can not be deduced from the extremal Lyapunov exponent.

 Take an i.i.d. random iteration 
of two Lipschitz maps $f_{1}$ and $f_{2}$ with Lipschitz constants equal to $2$ and $\frac{1}{2}$, respectively. If we denote $p_{i}\eqdef\mathbb{P}(X_{0}=i)>0$, $i=1,2$, we have
$$
\lambda_{0}=p_{1}\cdot\log 2+p_{2}\cdot \log \frac{1}{2}=(p_{1}-p_{2})\cdot \log 2.
$$
Then, $\lambda_{0}$ is negative, if and only if, $p_{1}< p_{2}$. 
On the other hand, if the maps are $J$-monotone and the induced i.i.d. random iteration satisfies the $J$-splitting condition (if the images of $f_{1}$ and $f_{2}$ are disjoint for instance), then Theorem \ref{sync} applies independently of the signal of $p_{1}-p_{2}$.

\subsubsection{Exponentially fast convergence in the Wasserstein distance}
Under assumptions of Theorem \ref{sync}, we are able to improve item (ii) of Theorem \ref{pullback} by showing exponentially fast convergence in the Wasserstein distance. Recall that for any pair of probability measures $\mu_{1}$ and $\mu_{2}$ on $S$ of bounded supports, the $1$-Wasserstein distance $W_{1}$ is given by 
\[
W_{1}(\mu_{1},\mu_{2})=\sup \left|\int f\, d\mu_{1}-\int f\, d\mu_{2}\right|, 
\]
where the sup is taken over all Lipschitz maps with Lipschitz constant $1$.

\begin{corollary}\label{expofastW}
Under assumptions of Theorem \ref{sync}, we have the following:
\begin{enumerate}
\item[(i)] $\pi\mathbb{P}$ has bounded support and there is $m_{0}$ such that for every probability measure $\varsigma$ the probability measure $T^n\varsigma$ has bounded support for every $n\geq m_{0}$.
\item[(ii)]
There are $C$ and $r<1$ such that for every probability measure $\varsigma$ on $S$  
\[
W_{1}(T^n\varsigma, \pi\mathbb{P})\leq C r^{n}
\]
for every $n\geq m_{0} $ .
\end{enumerate}

\end{corollary}

An exponentially fast convergence is also obtained in \cite{Bhbook,BhLe} for the $n$-step transition probability of an i.i.d. random iteration of monotone maps (w.r.t. the weak componentwise order) satisfying their splitting condition. The authors consider a Kolmogorov type distance, and Assumption \ref{A2} is not required. 
\subsubsection{Forward random attractor}\label{forwardattractor}
We now return to the setting of Section \ref{Apullbak}. We show that the map $\pi$ of Corollary \ref{A1} is also a forward random attractor. Given an i.i.d. random iteration $\varphi$,  we recall that a \emph{forward random attractor} of $\varphi$ is a $\varphi$-invariant random compact set $\omega\mapsto K(\omega)$ such that for every $x\in S$,
$$
\lim_{n\to \infty}d(\varphi(n,\omega,x),K(\sigma^{n}(\omega))=0, 
$$  
 for $\mathbb{P}$-almost every $\omega$,
where $\sigma$ is the shift map on $\Omega$. See for instance Crauel and Scheutzow \cite{CrauelScheutzow}.
As a consequence of Theorem \ref{sync}, we have the following: 
\begin{corollary}\label{Cor2}
Let $\varphi$ be an i.i.d. random iteration of $J$-monotone maps satisfying the $J$-splitting condition. Let $\pi\colon \Omega\to S$ be the map as in Corollary \ref{Cor1}.
 Then,  
under Assumption \ref{A2}, we have that for every $x\in S$
$$
\lim_{n\to \infty}d(\varphi(n,\omega,x),\pi(\sigma^{n}(\omega))=0
$$  
for $\mathbb{P}$-almost every $\omega$.
\end{corollary}
It follows from Corollaries \ref{Cor1} and \ref{Cor2} that $\omega\mapsto \{\pi(\omega)\}$ is a pullback and a forward random attractor.

\subsection{Functional central limit theorem}\label{FCLT}
We now present a functional central limit theorem for the class of i.i.d. random iterations considered in Theorem \ref{sync}.

In the study of central limit theorems 
for homogeneous Markov chains, there are several results that reduce the problem  to the verification of some analytical condition on the associated transfer operator, see for instance \cite{BataCLT,Gordin, Wood}. In our paper, we obtain a  functional central limit theorem by solving the Poisson equation induced by the transition probability defined in \eqref{tp}. Indeed, a solution of the Poisson equation allows us to reduce the problem the martigale case. This technique came up with  Gordin and Lifsic in \cite{Gordin}.

To be more precise, we first recall the definitions of transfer operator and Poisson equation. Let $p$ be a transition probability on a measurable space. The \emph{transfer operator} $P$ induced by the transition probability $p$ is defined as follows: given a non-negative measurable map $f$, the action of $P$ in $f$ is a non-negative measurable map $Pf$ given by 
\begin{equation} \label{transferoperator}
Pf(x)=\int f(y)\,p(x,dy).
\end{equation}
For a measurable map $f$ not necessarily non-negative, we write $f=f^{+}-f^{-}$ as a difference of non-negative measurable maps and we define 
$$
Pf(x)=Pf^{+}(x)-Pf^{-}(x),
$$
 if $Pf^{+}(x)$ and $Pf^{-}(x)$ are both finite. 
 
Let $\mu$ be a stationary measure. Given a non-constant (a.e.) map $\phi\colon S\to \mathbb{R}$ with $\phi\in L^{2}(\mu)$ and $\int \phi\, d\mu=0$, the equation
$$
(I-P)\psi=\phi
$$ 
 is called the \emph{Poisson equation}. 

Let  $Z_{n}$ be an ergodic stationary Markov chain with transition probability $p$ and stationary measure $\mu$. The existence of a solution $\psi\in L^{2}(\mu)$ of the Poisson equation implies that a FCLT holds for the Markov chain $Z_{n}$ taking $\phi$ as observable, i.e., the process $Y_{n}$ given by 
\[
Y_{n}(t)=\displaystyle\frac{1}{\sigma\sqrt{n}}\sum_{j=0}^{[nt]}\phi(Z_{j}), \qquad 0\leq t<\infty
\]
  converges in distribution (weak-star convergence) to the Wiener measure in $D[0,\infty)$, where $D[0,\infty)$ is the space of real-valued right continuous function on $[0,\infty)$ having left limits endowed with the Skorohod topology and $\sigma\eqdef \int \psi^2\,\mu-\int (P\psi)^2\, d\mu>0$.
  
 The FCLT stated above is proved by reducing the problem to the martigale case.
 See \cite[Page 1340]{BhLe} for this reduction and Billingsley \cite[Theorem 18.3]{Patrick} for a FCLT for martigale differences.

\begin{mtheorem}\label{functionalclt}
Let $S$ be a connected closed subset of $\mathbb{R}^k$ and let  $(f,X)$ be an i.i.d. random iteration of  $J$-monotone continuous maps on $S$ satisfying the $J$-splitting condition.
Let $\mu$ be the unique stationary measure and consider a non-constant (a.e.) Lipschitz map  $\phi\colon S\to \mathbb{R}$  in $L^{2}(\mu)$ with $\int \phi\, d\mu=0$. Then, under Assumption \ref{A2} 
\begin{enumerate}
\item[(i)] there is $\psi\in L^{2}(\mu)$ such that $(I-P)\psi=\phi$.

\item [(ii)] If  $Z_{n}$ is a stationary Markov chain associated with $(f,X)$, then  the process $Y_{n}$ given by 
\[
Y_{n}(t)=\displaystyle\frac{1}{\sigma\sqrt{n}}\sum_{j=0}^{[nt]}\phi(Z_{j}), \qquad 0\leq t<\infty, 
\]
converges in distribution to the Wiener measure in $D[0,\infty)$, where  $\sigma=\int (\psi-P\psi)^2\, d\mu$.
\end{enumerate} 
\end{mtheorem}

\begin{remark}\emph{
It follows from Theorem \ref{sync} that the FCLT in Theorem \ref{functionalclt} holds under every initial distribution.}
\end{remark}

Theorem \ref{functionalclt} extends \cite[Theorem 3.1]{BhLe} and \cite[Theorem 4.2, Page 362]{Bhbook}.
Therein, for i.i.d. random iterations of non-decreasing maps on a closed subset $S\subset \mathbb{R}^{k}$ satisfying the splitting condition, the authors show that the Poisson equation $(I-P)\psi=\phi$ has a solution provided that  $\phi$  may be expressed as a difference of two non-decreasing bounded functions. Furthermore, they also obtain  item (ii) under every initial distribution. However, their proof works only for i.i.d. random iterations of non-decreasing maps. It worth mentioning that our dynamical systems approach allows us to prove the FCLT for i.i.d. random iterations of $J$-monotone maps (increasing and decreasing).

\section{Proof of Theorem \ref{pullback}} \label{p1}

 We start with a preliminary result. Let $\pi_{s}\colon \mathbb{R}^k\to \mathbb{R}$ be the natural projection $\pi_{s}(x)=x_{s}$, $s=1,\dots, k$.

\begin{theorem}\label{tcni}
Let $S$ be a connected closed subset of $\mathbb{R}^k$ and let  $(f,X)$ be an i.i.d. random iteration of $J$-monotone continuous maps on $S$ satisfying the $J$-splitting condition.
Then, there exists $0<r <1$  such that for every  finite Borel measure $\eta$ on $\mathbb{R}$ there is $C\geq 0$ such that
	\[
	\int \eta( \pi_{s}(f^{n}_{\omega}(S)))\, d\mathbb{P}(\omega)=\int \eta(\pi_{s}(f_{X_0}\circ \cdots \circ  f_{X_{n-1}}(S)))	\, d\mathbb{P}\leq r^n C
	\]
	for every $n\geq 0$ and every $s=1,\dots, k$.
\end{theorem}

This theorem is an important step of the proof of Theorem 1 and its proof is
inspired by the ideas in \cite{Yuri}.

\begin{proof}[Proof of Theorem \ref{tcni}]
For every $s\in\{1,\dots, k\}$, $x\in S$ and $n\geq 1$, 
define 
\[
\Sigma_{n}^x(s)=\{\omega\in \Omega\colon x\in \pi_{s}(f_{X_{0}(\omega)}\circ \cdots \circ  f_{X_{n-1}(\omega)}(S))\}.	
\] 
Note that for every $n\geq 1$, we have the following useful property 
$$
\Sigma_{n+1}^x(s)\subset \Sigma_{n}^x(s)
$$
for every $s=1,\dots, k$.

\begin{lemma}\label{weaki}
There are $\lambda<1$ and an integer $m\geq 1$  such that 
\[
\mathbb{P}(\Sigma_{jm}^x(s))\leq \lambda^j
\]
for every $j\geq 1$ and $s=1,\dots, k$.
\end{lemma}
\begin{proof}
Fix $s\in \{1,\dots, k\}$.
The proof is by induction on $j$. Let $m$,
$A$ and $B$ as in the definition of the $J$-splitting condition, recall Definition \ref{Definition1}.  For $j=1$, it follows from the $J$-splitting condition and Remark \ref{capempty} that
\[
\Sigma_{m}^x(s)\subset \Omega-\Gamma_{1}, 
\]
where either
\[ 
\Gamma_{1}=\{\omega\in \Omega\colon (X_{0}(\omega), \dots,X_{m-1}(\omega))\in A\}\eqdef A_{1}
\]
 or
 \[ 
\Gamma_{1}=\{\omega\in \Omega\colon (X_{0}(\omega), \dots,X_{m-1}(\omega))\in B\}\eqdef B_{1}.
\] 
In particular, 
\[
\mathbb{P}(\Sigma_{m}^x(s))\leq 1-\rho\eqdef\lambda, 
\]
where $\rho=\min\{\mathbb{P}(A_{1}),\mathbb{P}(B_{1})\}>0$.

Assume that the lemma holds for $j\geq 1$. We now prove the lemma for $j+1$. Consider the random variable $Z=(X_{0},\dots, X_{jm-1})$ taking values in $E^{jm}$. For every $z\in E^{jm}$, we claim that
\begin{equation}\label{funddesi}
\Sigma_{(j+1)m}^x(s)\cap [Z=z]\subset \Sigma_{jm}^x(s)-[Z=z]\cap\Gamma_{j}\cap \Sigma_{jm}^x(s)
\end{equation}
where  either
$$
\Gamma_{j}=\{\omega \in \Omega\colon (X_{jm}(\omega),\dots, X_{(j+1)m-1}(\omega))\in A\}\eqdef A_{j}
$$
or 
$$
\Gamma_{j}=\{\omega\in \Omega\colon (X_{jm}(\omega),\dots, X_{(j+1)m-1}(\omega))\in B\}\eqdef B_{j}
$$
Indeed, if \eqref{funddesi} does not hold, then there are
$\omega, \hat\omega\in \Omega$ such that 
\[
\omega \in \Sigma_{(j+1)m}^x(s) \cap [Z=z]\cap A_{j} \quad \mbox{and} \quad \hat \omega \in \Sigma_{(j+1)m}^x(s) \cap [Z=z]\cap B_{j}.
\]
If we write $z=(\alpha_{0},\dots, \alpha_{jm-1})$, we have that 
\[
x\in \pi_{s}(f_{\alpha_{0}}\circ \dots\circ  f_{\alpha_{jm-1}}\circ f_{X_{mj}(\omega)}\circ \dots \circ f_{X_{(j+1)m}(\omega)}(S))
\] 
and
\[ 
x\in \pi_{s}(f_{\alpha_{0}}\circ \dots\circ  f_{\alpha_{jm-1}}\circ f_{X_{mj}(\hat \omega)}\circ \dots \circ f_{X_{(j+1)m}(\hat \omega)}(S)). 
\]
 The inclusions above can not hold simultaneously. Again, this follows from the $J$-splitting condition and Remark \ref{capempty}. This proves that \ref{funddesi} holds. 

\vspace{2mm}

Now, let $\mathbb{P}(\cdot|Z=\cdot)$ denote the  regular conditional probability given by $Z$. Since 
$\mathbb{P}([Z=z]|Z=z)=1$, it follows from \eqref{funddesi} that
\[
\begin{split}
\mathbb{P}(\Sigma_{(j+1)m}^x(s)|Z=z)\leq 
\mathbb{P}(\Sigma_{jm}^x(s)|Z=z)-\mathbb{P}( \Gamma_{j}\cap \Sigma_{jm}^x(s)) |Z=z).
\end{split}
\]
In particular, since $z$ is arbitrary, we have 
\begin{equation}\label{cregular}
\mathbb{P}(\Sigma_{(j+1)m}^x(s)|Z=z)\circ Z\leq 
\mathbb{P}(\Sigma_{jm}^x(s)|Z=z)\circ Z-\mathbb{P}( \Gamma_{j}\cap \Sigma_{jm}^x(s)) |Z=z)\circ Z.
\end{equation}
It follows from the definition of the regular conditional probability that
$$
\mathbb{P}( \Gamma_{j}\cap \Sigma_{m}^x(s)) |Z=z)\circ Z=\mathbb{E}(\mathds{1}_{ \Gamma_{j}}\mathds{1}_{ \Sigma_{jm}^x(s)}|Z).
$$
Since $\Gamma_{j}$ and $Z$ are independent random variables, we conclude
\[
\mathbb{E}(\mathds{1}_{ \Gamma_{j}}\mathds{1}_{ \Sigma_{jm}^x(s)}|Z)=\mathds{1}_{ \Gamma_{j}}\mathbb{E}(\mathds{1}_{ \Sigma_{jm}^{x}(s)}|Z)=\mathds{1}_{\Gamma_{j}}\mathbb{P}(\Sigma_{jm}^{x}(s)|Z=z)\circ Z.
\]
Note that $\mathds{1}_{\Gamma_{j}}$ and 
$\mathbb{P}(\Sigma_{jm}^{x}(s)|Z=z)\circ Z$ are also independent random variables. Therefore, 
integrating \eqref{cregular} we get
\begin{equation}\label{md}
\mathbb{P}(\Sigma_{(j+1)m}^x(s))\leq \mathbb{P}(\Sigma_{jm}^x(s))-\mathbb{P}(\Gamma_{j})\mathbb{P}(\Sigma_{jm}^x(s))=\mathbb{P}(\Sigma_{jm}^x(s))(1-\mathbb{P}(\Gamma_{j})).
\end{equation}
Since $X=\{X_{n}\}$ is an i.i.d. sequence of random variables, we have $\mathbb{P}(A_{j})=\mathbb{P}(A_{1})$ and $\mathbb{P}(B_{j})=\mathbb{P}(B_{1})$. Hence, we conclude from \eqref{md} that 
\[
\mathbb{P}(\Sigma_{(j+1)m}^x(s))\leq \lambda^j\cdot(1-\rho)=\lambda^{j+1}.
\]
Note that $\lambda$ does not depend on $s$. The proof of the lemma is now complete.
\end{proof}

We now conclude the proof of Theorem \ref{tcni}. Let $\lambda$ as in Lemma \ref{weaki} and define $\hat \lambda=\lambda^{\frac{1}{m}}$. Thus, for every $s$ we have 
\[
\mathbb{P}(\Sigma_{jm}^x(s))\leq \hat\lambda^{mj}.
\] 
Now, choose any $r>0$ such that $r<1$ and $\hat \lambda \leq \min\{r,r^2,\dots r^{m}\}$.
Let $n\geq m$. Then, there is $e\in\{0,\dots, m-1\}$ such that $n=jm+e$. Note that $\Sigma_{n}^x(s)\subset\Sigma_{jm}^x(s)$. Therefore
\[
\mathbb{P}(\Sigma_{n}^x(s))\leq \mathbb{P}(\Sigma_{jm}^x(s))\leq \hat\lambda^{mj}=\hat \lambda^{mj-1}\hat \lambda\leq r^{mj-1}r^{e+1}=r^{n}.
\]
Now, it follows from Fubbini theorem that 
\[
\begin{split}
\int \eta(\pi_{s}(f_{\omega}^{n}(S)))\, d\mathbb{P}(\omega)
&=\int \mathbb{P}(x\in \pi_{s}(f^{n}_{(\cdot)}(S)))\, d\eta(x)\\
&=\int \mathbb{P}(\Sigma_{n}^{x}(s))\, d\eta(x)\leq r^{n}\eta(\mathbb{R})
\end{split}
\]
for every $n\geq m$.
This implies that there is $C\geq 0$ such that 
\[
\begin{split}
\int \eta(\pi_{s}(f_{\omega}^{n}(S)))\, d\mathbb{P}(\omega)
&\leq r^{n}C
\end{split}
\]
for every $n\geq 0$ for every $s$.
\end{proof}

Before proving Theorem \ref{pullback}, we need two technical lemmas. The first one
says that with probability $1$ the set $\pi_{s}(f_{X_{0}}\circ \cdots \circ  f_{X_{n-1}}(S))$ is bounded for $n$ sufficiently large. The second one is a general result from measure theory that will be used to state that  $\pi_{s}(f_{X_{0}}\circ \cdots \circ  f_{X_{n-1}}(S))$ is ``contracting'' exponentially fast with respect to any finite Borel measure.

\begin{lemma}\label{A1}
For $\mathbb{P}$-almost every $\omega$, there is $n_{0}$ (depending on $\omega$) such that 
$$
\pi_{s}(f_{X_{0}(\omega)}\circ \cdots \circ  f_{X_{n-1}(\omega)}(S))
$$
is bounded for every $n\geq n_{0}$.

\end{lemma}
\begin{proof}
Let $A,B\subset E$ and $m$ as in the definition of the $J$-splitting condition.
Consider the set $A\times B$. For every $(\alpha,\beta)$ we define the sets 
$$
G_{\alpha\beta}= \pi_{s}(f_{\alpha_{0}}\circ\dots \circ  f_{\alpha_{m-1}}\circ  f_{\beta_{0}} \circ \dots \circ f_{\beta_{m-1}}(S))
$$
and 
$$
G_{\alpha\alpha}=\pi_{s}(f_{\alpha_{0}}\circ\dots \circ  f_{\alpha_{m-1}}\circ  f_{\alpha_{0}} \circ \dots \circ f_{\alpha_{m-1}}(S))
$$ 
and 
$$
G_{\beta\alpha}=\pi_{s}(f_{\beta_{0}} \circ \dots \circ f_{\beta_{m-1}}\circ f_{\alpha_{0}}\circ\dots \circ  f_{\alpha_{m-1}}(S)).
$$
 Associate to theses set there are the subsets of $A\times B$: 
$$
E_{\alpha\beta}=\{(\alpha,\beta)\in A\times B\colon G_{\alpha\beta}\,\, \mbox{is bounded}\}
$$
and 
$$
E_{\alpha\alpha}=\{(\alpha,\beta)\in A\times B\colon G_{\alpha\alpha}\,\, \mbox{is bounded}\}
$$
and 
$$
E_{\beta\alpha}=\{(\alpha,\beta)\in A\times B\colon G_{\beta\alpha}\,\, \mbox{is bounded}\}.
$$

\begin{claim}
We have 
$$
A\times B=E_{\alpha\beta}\cup E_{\alpha\alpha}\cup E_{\beta\alpha}.
$$

\end{claim}
\begin{proof}
Given $(\alpha,\beta)=(\alpha_{0},\dots, \alpha_{m-1},\beta_{0},\dots, \beta_{m-1})\in A\times B$, it follows from the $J$-splitting condition that 
$$
f_{\alpha_{0}}\circ\dots \circ  f_{\alpha_{m-1}}(S)<_{J} f_{\beta_{0}} \circ \dots \circ f_{\beta_{m-1}}(S).
$$ 
In particular, the sets $G_{\alpha\beta}, G_{\alpha\alpha},G_{\beta\alpha}$ are disjoints. Since $S$ is connected, they also are intervals, which implies that some of them must be bounded. 
\end{proof}

Now, since $\nu^{2m}(A\times B)>0$, we have that at least one of the sets $G_{\alpha\beta}, G_{\alpha\alpha},G_{\beta\alpha}$ has positive $\nu^{2m}$-measure. Without loss of generality, we can assume that $\nu^{2m}(G_{\alpha\beta})>0$.   In particular, it follows from Birkoff's Ergodic Theorem that  for $\mathbb{P}$-almost every $\omega$, there is $n_{1}$ such that 
$$
(X_{n_{1}}(\omega),\dots,X_{n_{1}+2m-1}(\omega))\in G_{\alpha\beta}.
$$
By definition of $G_{\alpha\beta}$, we have that $\pi_{s}(f_{X_{n_{1}}(\omega)}\circ \dots \circ f_{X_{n_{1}+2m-1}(\omega)}(S))$ is bounded. 

Since $S$ is closed, the image of a bounded set by a continuous maps is also bounded and then we conclude that for every $n\geq n_{0}\eqdef n_{1}+2m$ we have that the set
$$
\pi_{s}(f_{X_{0}(\omega)}\circ \dots \circ f_{X_{n}(\omega)}(S))
$$
is bounded.
\end{proof}

\begin{lemma}\label{mt}
	Let $Y_n\colon \Omega\to [0,\infty)$ be a sequence of measurable maps and assume that there exists  $0<\lambda<1$ 
		such that  $\mathbb{E}Y_{n} \leq \lambda^n$.  Then, there exist an integrable function $c\colon \Omega\to [0,\infty)$ and $q<1$  such that for $\mathbb{P}$-almost every $\omega$ it holds $Y_n(\omega)\leq c(\omega)\cdot q^n$ 
		for every $n\geq 0$  .
\end{lemma}
\begin{proof}
We take  any $q<1$ with $\lambda<q$ and apply the 
Monotone Convergence Theorem to obtain that
$$
\displaystyle\int\sum _{n=1}^{\infty}\frac{ Y_{n}(\omega)}{q^{n}}\, d\mathbb{P}(\omega)=
\displaystyle\sum _{n=1}^{\infty}\frac{\mathbb{E}Y_{n}}{q^{n}}<\infty.
$$ 
Therefore
$$
c(\omega)\eqdef\sum _{n=0}^{\infty}\frac{ Y_{n}(\omega)}{q^{n}}< \infty,
$$
 for 
$\mathbb{P}$-almost every $\omega$, which implies that 
$$
Y_{n}(\omega)\leq c(\omega) q^{n}
$$ 
$\mathbb{P}$-almost every $\omega$.
\end{proof}

\subsection{Proof of Theorem \ref{pullback}}
Denote by $\mbox{Leb}$ the Lebesgue measure on $\mathbb{R}$ and consider the finite Borel measure $m_{\ell}$ given by $m_{\ell}(A)=\mbox{Leb}([-\ell,\ell]\cap A)$. Thus, it follows from Theorem \ref{tcni} and Lemma \ref{mt} that for every $s$, there is an integrable map $c\colon \Omega\to [0,\infty)$ (which we can assume that does not depend on $s$, because $s$ varies on a finite set) such that for $\mathbb{P}$-almost every $\omega$ it holds
\begin{equation}\label{T31}
m_{\ell}(\pi_{s}(f_{X_{0}(\omega)}\circ \cdots \circ  f_{X_{n-1}(\omega)}(S)))\leq c(\omega)q^n
\end{equation}
for every $n\geq 1$. 

Now, let $\omega$ as in Lemma \ref{A1} satisfying eq. \eqref{T31}. Then, there is $n_{0}$ such that  
$$
\pi_{s}(f_{X_{0}(\omega)}\circ \dots \circ f_{X_{n-1}(\omega)}(S))
$$
is bounded for every $n\geq n_{0}$ and $s$. Since the sequence $(\pi_{s}(f_{X_{0}(\omega)}\circ \dots \circ f_{X_{n-1}(\omega)}(S)))_{n}$ is nested, we conclude that there is $\ell\in \mathbb{N}$ such that 
$$
\pi_{s}(f_{X_{0}(\omega)}\circ \dots \circ f_{X_{n-1}(\omega)}(S))\subset [-\ell,\ell]
$$
for every $n\geq n_{0}$ and $s$.


Note that for every connected subset $I$ of $[-\ell,\ell]$ we have $m_{\ell}(I)=\mbox{diam}(I)$. Since  $\pi_{s}(f_{X_0(\omega)}\circ \cdots \circ  f_{X_{n-1}(\omega)}(S))$ is a connected subset of $\mathbb{R}$, we get that 
$$
m_{\ell}(\pi_{s}(f_{X_0(\omega)}\circ \cdots \circ  f_{X_{n-1}(\omega)}(S)))=\mbox{diam}(\pi_{s}(f_{X_0(\omega)}\circ \cdots \circ  f_{X_{n-1}(\omega)}(S)))
$$
for every $n\geq n_{0}$ and $s$.
Therefore, it follows from eq. \eqref{T31} that 
$$
\mbox{diam}(\pi_{s}(f_{X_0(\omega)}\circ \cdots \circ  f_{X_{n-1}(\omega)}(S)))\leq c(\omega)q^n
$$
for every $n\geq n_{0}$ and $s$.

Now, if $\mathbb{R}^{k}$ is endowed with the sup distance, we have
$$
\mbox{diam}\, f_{X_0(\omega)}\circ \cdots \circ  f_{X_{n-1}(\omega)}(S)=\max_{s}\mbox{diam}(\pi_{s}(f_{X_0(\omega)}\circ \cdots \circ  f_{X_{n-1}(\omega)}(S)))\leq c(\omega)q^n
$$
for every $n\geq n_{0}$.
This implies that 
\begin{equation}\label{cauchy}
\lim_{n\to \infty}\mbox{diam}\, f_{X_0(\omega)}\circ \cdots \circ  f_{X_{n-1}(\omega)}(S)=0.
\end{equation}
To conclude the proof of item (i), note that it follows from \eqref{cauchy} that for every $x$ the sequence $f_{X_{0}(\omega)}\circ\dots\circ f_{X_{n}(\omega)}(x)$ is a Cauchy sequence.

Item (ii) follows from Letac principle \cite{Letac} and item (i).
  \hfill \qed

\subsection{Proof of Corollary \ref{Cor1}}
   Since the sequence $\hat X_{n}(\omega)=\omega_{-n}$, $n\geq 1$, is an i.i.d. sequence of random variables with distribution $\nu$, it follows from Theorem \ref{pullback}  that for $\mathbb{P}$-almost every $\omega$, the limit 
$$
\lim_{n\to \infty} f_{\omega_{-1}}\circ \dots\circ f_{\omega_{-n}}(x)\eqdef \pi(\omega)
$$  
exists and  is independent of the point $x$. 
  \hfill \qed

\section{Proof of Theorem \ref{sync}} \label{p2}
In this section, we consider $\mathbb{R}^{k}$ endowed with the taxcab metric.  Then, we have for every subset $B\subset\mathbb{R}^{k}$ that 
 \begin{equation} \label{denovo}
 \mbox{diam}\, B\leq \sum_{s=1}^{k}\mbox{diam}\,\pi_{s}(B)
 \end{equation}
 By Assumption \ref{A2}, there are a bounded set $B$ and $m_{0}$ such that for $\mathbb{P}$-almost every $\omega$ we have
$$
 f_{\omega}^{m_{0}}(S)\subset B. 
$$
In particular, since the random variables $(X_{0},\dots, X_{m_{0}-1})$ and $(X_{m_{0}-1},\dots, X_{0})$ have the same distribution, we also have 
$$
f_{X_{0}(\omega)}\circ \dots \circ f_{X_{m_{0}-1}(\omega)}(S)\subset B
$$
for $\mathbb{P}$-almost every $\omega$. This implies that 
for every $n\geq m_{0}$, it holds 
$$
f_{X_{0}(\omega)}\circ \dots \circ f_{X_{n-1}(\omega)}(S)\subset B
$$
for $\mathbb{P}$-almost every $\omega$. Indeed, recall that the sequence $(f_{X_{0}(\omega)}\circ \dots \circ f_{X_{n-1}(\omega)}(S))_{n}$ is nested.

Now, let $\ell$ be such that $\pi_{s}(B)\subset [-\ell,\ell]$ for every $s$. In particular, 
$$
\pi_{s}(f_{X_{0}(\omega)}\circ \dots \circ f_{X_{n-1}(\omega)}(S))\subset [-\ell,\ell]
$$
for every $n\geq m_{0}$ and $s$. Let $m_{\ell}$ be the finite Borel measure given by $m_{\ell}(A)=\mbox{Leb}([-\ell,\ell]\cap A)$. Note that if $I$ is a connected set with $I\subset[-\ell, \ell]$, then 
$m_{\ell}(I)=\mbox{diam}\, I$. Then, applying Theorem \ref{tcni} to the measure $m_{\ell}$, 
we get that there is $C$ such that
\begin{equation}\label{repe}
\begin{split}
	\int \mbox{diam}\,(\pi_{s}(f^{n}_{\omega}(S)))\, d\mathbb{P}(\omega)&=\int \mbox{diam}(\pi_{s}(f_{X_0}\circ \cdots \circ  f_{X_{n-1}}(S)))	\, d\mathbb{P}\\
&=\int m_{\ell}(\pi_{s}(f_{X_0}\circ \cdots \circ  f_{X_{n-1}}(S)))	\, d\mathbb{P}\leq r^n C
	\end{split}
	\end{equation}
for every $n\geq m_{0}$ and every $s$, where the first equality follows from the fact the $(X_{0},\dots, X_{n-1})$ and $(X_{n-1},\dots, X_{0})$ have the same distribution. 
	
	Thus, it follows from eq. \eqref{denovo} and \eqref{repe}, that 
	$$
	\int \mbox{diam}\,(f^{n}_{\omega}(S))\, d\mathbb{P}(\omega)=\int \mbox{diam}(f_{X_0}\circ \cdots \circ  f_{X_{n-1}}(S))	\, d\mathbb{P}\leq kC r^n
	$$
	for every $n\geq m_{0}$. Now, it follows from Lemma \ref{mt} that there is an integrable map $c\colon \Omega\to \mathbb{R}$ such that for $\mathbb{P}$-almost every $\omega$ we have
	$$
	\mbox{diam}\,(f^{n}_{\omega}(S))\leq c(\omega)r^{n}
	$$
	for every $n\geq m_{0}$.
	
	\hfill \qed

\subsection{Proof of Corollary \ref{expofastW}} 
Arguing as in the proof of Theorem \ref{sync}, we have that there are $m_{0}\geq 1$  and a bounded set such that 
for every $n\geq m_{0}$, it holds 
\begin{equation}\label{ql}
f_{X_{0}(\omega)}\circ \dots \circ f_{X_{n-1}(\omega)}(S)\subset B
\end{equation}
for $\mathbb{P}$-almost every $\omega$. This implies that $\mbox{supp}\,\pi\mathbb{P}\subset \bar B$.
 Note that $\mathbb{P}(\{\omega\colon f_{\omega}^{n}(x)\,\in B\})= \mathbb{P}(\{\omega\colon f_{X_{0}(\omega)}\circ \dots \circ f_{X_{n-1}(\omega)}(x)\,\in B\})=1$, for every $n\geq 1$. Also, by the definition of $T$ we have that 
 $$
T^{n}\varsigma(B)=\int \mathbb{P}(\{\omega\colon f_{\omega}^{n}(x)\,\in B\})\, d\varsigma(x)
$$ 
for every probability measure $\varsigma$ on $S$.
Therefore, it follows from eq. \eqref{ql} that
$$
T^{n}\varsigma(B)=1
$$
for every $n\geq m_{0}$. This implies that $\mbox{supp}\, T^{n}\varsigma\subset \bar B$ for every $n\geq m_{0}$.

 We now prove item (ii).
Let $\{X_{n}\}_{n\in \mathbb{Z}}$ be a bilateral sequence of i.i.d. random variables with distribution $\nu$. For every $n\geq 0$, the sequence $X_{n-1},\dots X_{0},X_{-1},\dots$ is an i.i.d. sequence of random variables with distribution $\nu$. Then, it follows from Theorem \ref{pullback}, that for every $n\geq 0$ there is a measurable map $ \pi_{n}$ such that 
$$
 \pi_{n}(\omega)=  \lim_{k \to \infty}f_{X_{n-1}(\omega)}\circ \dots \circ f_{X_{0}(\omega)}\circ f_{X_{-1}(\omega)}\circ \dots \circ f_{X_{-k}(\omega)}(p)
$$
for $\mathbb{P}$-almost every $\omega$. Note that $\pi_{n}(\omega)=f_{\omega}^{n}(\pi_{0}(\omega))$ and for every $n$, and the maps $\pi$ and
$\pi_{n}$ have the same distribution. In particular, 
for every $\pi\mathbb{P}$-integrable map $\phi\colon S\to \mathbb{R}$  we have 
$$
\int \phi(\pi(\omega))\,d\mathbb{P}(\omega)=\int \phi(f_{\omega}^{n}(\pi_{0}(\omega)))\,d\mathbb{P}(\omega).
$$
Assume now that $\phi$ is  a Lipschitz map with Lipschitz constant $1$. It follows from the definition of the transfer operator $P$ and the Markov operator $T$ that,
$$
P^n\phi(x)=\int \phi(f^{n}_{\omega}(x))\, d\mathbb{P}(\omega) \quad \mbox{and}\quad \int P^n \phi(x)\, d\, \mu(x)=\int \phi(x)\, dT^{n}\varsigma(x),
$$
for every $n\geq m_{0}$ and $x\in S$. Hence,
we have
\begin{equation}\label{maeq}
\begin{split}
\left|\int \phi\, dT^{n}\varsigma-\int \phi\,d\pi\mathbb{P}\right|
&=\left|\int P^n\phi(x)\, d\mu(x)-\int \phi(f_{\omega}^{n}(\pi_{0}(\omega)))\,d\mathbb{P}(\omega)\right|\\
&= \left|\int \int \phi(f_{\omega}^{n}(x))\, d\mathbb{P}(\omega)\, d\varsigma(x) -\int \phi(f_{\omega}^{n}(\pi_{0}(\omega)))\,d\mathbb{P}(\omega)\right|\\
&=\left|\int \int \left(\phi(f_{\omega}^{n}(x))-\phi(f_{\omega}^{n}(\pi_{0}(\omega)))\right)\, d\mathbb{P}(\omega)\, d\varsigma(x)\right|.
\end{split}
\end{equation}
It follows from Theorem \ref{sync} that there are an integrable map $c\colon \Omega\to \mathbb{R}$ and $r<1$ such that 
$$
 \mbox{diam}\,f_{\omega}^n(S)\leq c(\omega)r^{n}
$$
for $\mathbb{P}$-almost every $\omega$ and $n\geq m_{0}$. 
Then, 
$$
\left|\phi(f_{\omega}^{n}(x))-\phi(f_{\omega}^{n}(\pi_{0}(\omega)))\right|\leq \mbox{diam}\,f_{\omega}^n(S)\leq c(\omega)r^n
$$
for $\mathbb{P}$-almost every $\omega$ and $n\geq m_{0}$,
and thus it follows from eq. \eqref{maeq} that  

$$
\left|\int \phi\, dT^{n}\mu-\int \phi\,d\pi\mathbb{P}\right|\leq \int \mbox{diam}\,f_{\omega}^n(S)\, d\mathbb{P}(\omega)\leq Cr^n
$$
for every $n\geq m_{0}$. This implies that 
$$
W_{1}(T^{n}\mu,\pi\mathbb{P})\leq C r^{n}
$$
for every $n\geq m_{0}$.
\hfill \qed

\subsection{Proof of Corollary \ref{Cor2}}
Using the notation of Section \ref{forwardattractor}, Theorem \ref{sync} can be rewrite as: there are $c\colon \Omega\to \mathbb{R}$, $r<1$ and a constant $C\geq 0$ such that for $\mathbb{P}$-almost every $\omega$,
\begin{equation}\label{RW}
d(\varphi(n,\omega,x),\varphi(n,\omega,y))\leq c(\omega)r^{n}
\end{equation}
for every $x,y\in S$ and $n\geq m_{0}$.
 
Since the sequence $\hat X_{n}(\omega)=\omega_{-n}$, $n\geq 1$, is an i.i.d. sequence with distribution $\nu$, we get  from Theorem \ref{pullback}  that for $\mathbb{P}$-almost every $\omega$, the limit 
$$
\lim_{n\to \infty} f_{\omega_{-1}}\circ \dots\circ f_{\omega_{-n}}(x)\eqdef \pi(\omega)
$$  
exists and is independent of the point $x$. 
Note that the map $\pi$ satisfies the following invariance equation 
$$
f_{\omega}(\pi(\omega))=\pi(\sigma(\omega))
$$
for $\mathbb{P}$-almost every $\omega$, where $\sigma$ is the shift map on $E^{\mathbb{Z}}$. By induction, we get 
$$
\varphi(n,\omega,\pi(\omega))=\pi(\sigma^{n}(\omega))
$$
for $\mathbb{P}$-almost every $\omega$ and every $n$. Therefore, 
\[
\begin{split}
\lim_{n\to \infty}d(\varphi(n,\omega,x),\pi(\sigma^{n}(\omega))&=\lim_{n\to \infty}d(\varphi(n,\omega,x),\varphi(n,\omega,\pi(\omega)))\\
&\leq 
\lim_{n\to \infty}c(\omega)r^{n}=0 
\end{split}
\] 
for $\mathbb{P}$-almost every $\omega$.
\hfill \qed

\section{Solving the Poisson equation. Proof of Theorem \ref{functionalclt}}\label{p3}
As observed in Section \ref{FCLT}, we only need to prove item (i). 
Let $(f,X)$ be an i.i.d. random iteration as in Theorem \ref{functionalclt} and $p$ be the  transition probability given by \eqref{tp}.  Let $P$ be the associated transfer operator as defined in \ref{transferoperator}. 
Let $\mu$ be the unique stationary measure and consider 
a Lipschitz map $\phi\in L^{2}(\mu)$ with $\int \phi\, d\mu=0$.
It follows from the definition of the transfer operator $P$ that for every $x$ and every $n\geq 1$ we have
\[
P^n\phi(x)=\int \phi(f^{n}_{\omega}(x))\, d\mathbb{P}(\omega) \quad \mbox{and}\quad \int P^n \phi(x)\, d\, \mu(x)=\int \phi(x)\, dT^{n}\mu(x),
\]
 where $T$ is the Markov operator. Then, for every $x$ we have 
\[
\begin{split}
|P^{n}\phi(x)|&=\left|P^{n}\phi(x)-\int\phi(y)\, d\mu(y)\right|\\
&=\left|P^n\phi(x)-\int P^n\phi(y)\,d\mu(y)\right|\\
&=\left|\int P^n\phi(x)\,d\mu(y)-\int P^n\phi(y)\,d\mu(y)\right| \\
&=\left|\int \int \phi(f^{n}_{\omega}(x))\, d\mathbb{P}(\omega) \, d\mu(y)-\int \int \phi(f^{n}_{\omega}(x))\, d\mathbb{P}(\omega) \, d\mu(y)\right|\\
&\leq \int \int \left|\phi(f^{n}_{\omega}(x))-\phi(f^{n}_{\omega}(y))\right|\, d\mathbb{P}(\omega)\, d\mu(y). 
\end{split}
\]
Let $L$ be the Lipschitz constant of $\phi$. Thus, we have
\[
\left|\phi(f^{n}_{\omega}(x))-\phi(f^{n}_{\omega}(y))\right|\leq L \cdot d(f^{n}_{\omega}(x),f^{n}_{\omega}(y))\leq L\cdot \mbox{diam}f_{\omega}^{n}(S)
\]
for every $\omega\in \Omega$.
On the other hand, it follows from Theorem \ref{sync} that there are constants $C\geq 0$, $0<r<1$ and an integer $m_{0}\geq 1$ such that  
$$
\int \mbox{diam} f_{\omega}^{n}(S)\, d\mathbb{P}(\omega)\leq C r^{n}
$$
for every $n\geq m_{0}$.

Therefore, for every $x$ and $n\geq m_{0}$ we have
$$
(P^{n}\phi(x))^2\leq C^2\cdot L^2\cdot r^{2n}.
$$
This implies that $\|P^{n}\phi\|_{2}\leq C_{0}\lambda^{n}$ for every $n\geq m_{0}$, where $C_{0}=
C^2\cdot L^2$ and $\lambda=r^2<1$. In particular, 
$$
\sum_{n=0}^{\infty} \|P^{n}\phi\|_{2}<\infty.
$$
Then, the map $\psi=-\sum_{n=0}^{\infty} P^{n}\phi$ is a well defined element of $L^{2}(\mu)$ and solves the Poisson equation $(I-P)\psi=\phi$.
\hfill \qed

\bibliographystyle{acm}

\end{document}